\documentclass[a4paper,11pt,reqno]{amsart}
\usepackage{a4wide}
\usepackage{enumerate}
\usepackage{amssymb, amsmath}
\usepackage{mathrsfs}
\usepackage{amscd}
\usepackage[active]{srcltx}
\usepackage{fancyvrb,color}
\usepackage{verbatim}
\usepackage[colorlinks,linkcolor={black},citecolor={blue},urlcolor={black}]{hyperref}
\definecolor{darkgreen}{rgb}{0.00,0.33,0.25}
\definecolor{darkred}{rgb}{0.60,0.05,0.05}
\definecolor{darkblue}{rgb}{0.05,0.05,0.60}

\usepackage{tikz}

\theoremstyle{plain}
\newtheorem{theorem}{Theorem}[section]

\newtheorem{lemma}[theorem]{Lemma}

\newtheorem{definition}[theorem]{Definition}

\newtheorem*{definition*}{Definition}

\theoremstyle{remark}
\newtheorem{remark}[theorem]{Remark}

\newtheorem*{claim*}{Claim}
\newtheorem*{remark*}{Remark}
\newtheorem*{example*}{Example}
\newtheorem*{notation*}{Notation}

\numberwithin{equation}{section}

\def\R{{\mathbb R}}

\renewcommand{\a}{\alpha}

\newcommand{\eps}{\varepsilon}
\renewcommand{\phi}{\varphi}

\newcommand{\hL}{\widehat L}
\newcommand{\dd}{\; \mathrm{d}}

\newcommand{\ip}[1]{\langle {#1}\rangle}
\newcommand{\bip}[1]{\big\langle {#1}\big\rangle}

\newcommand{\abs}[1]{\vert {#1}\vert}

\DeclareMathOperator{\Ric}{Ric}

\DeclareMathOperator{\Hess}{Hess}

\newcommand{\ddt}{\frac{\mathrm{d}}{\mathrm{d}t}}

\newcommand{\cH}{\mathcal{H}}
\newcommand{\cB}{\mathcal{B}}

\newcommand{\cW}{\mathcal{W}}

\newcommand{\cS}{\mathcal{S}}

\newcommand{\sD}{\mathsf{D}}

\newcommand{\cX}{\mathcal{X}}
\newcommand{\cE}{\mathcal{E}}
\newcommand{\cA}{\mathcal{A}}
\newcommand{\cY}{\mathcal{Y}}

\newcommand{\cP}{\mathscr{P}}
\newcommand{\PX}{\cP(\cX)}
\newcommand{\PXs}{\cP_*(\cX)}
\newcommand{\hrho}{\hat\rho}

\definecolor{jan}{rgb}{0.0,0.3,0.8}

\begin{document}

\title[Ricci bounds for Bernoulli--Laplace and Random Transposition models]
{Discrete Ricci Curvature bounds for Bernoulli--Laplace and  Random Transposition models}

\author{ {Matthias Erbar}, {Jan Maas}, {Prasad Tetali} }
\thanks{P.T. gratefully acknwoledges support by the NSF grants
  DMS-1101447 and DMS-1407657. E.M and J.M gratefully acknowledge
  support by the German Research Foundation through the Collaborative
  Research Center 1060 \emph{The Mathematics of Emergent Effects} and
  the Hausdorff Center for Mathematics. This material is based upon
  work supported by the National Science Foundation under Grant
  No. 0932078 000, while E.M and J.M were in residence at MSRI, in the
  fall of 2013. All three authors thank the Mathematical Sciences
  Research Institute (MSRI) and the Simons Institute for the Theory of
  Computing, Berkeley, CA, for the hospitality and the conducive
  atmosphere of these institutes which facilitated this research
  collaboration. }

\address{Matthias Erbar,
Institute for Applied Mathematics\\
University of Bonn\\
Endenicher Allee 60\\
53115 Bonn\\
Germany} 
\email{erbar@iam.uni-bonn.de}

\address{Jan Maas,
Institute for Applied Mathematics\\
University of Bonn\\
Endenicher Allee 60\\
53115 Bonn\\
Germany} 
\email{maas@uni-bonn.de}

\address{Prasad Tetali,
School of Mathematics\\ 
Georgia Institute of Technology\\ 
Atlanta\\ GA 30332\\ USA}
\email{tetali@math.gatech.edu}

\date\today


 \begin{abstract}
   We calculate a Ricci curvature lower bound for some classical
   examples of random walks, namely, a chain on a slice of the
   $n$-dimensional discrete cube (the so-called Bernoulli--Laplace
   model) and the random transposition shuffle of the symmetric group
   of permutations on $n$ letters.
 \end{abstract}

\maketitle

\tableofcontents

\section{Introduction}

Many analytic and probabilistic properties of diffusion processes can
be derived from geometric properties of the underlying space. In
particular, a positive lower bound on the \emph{Ricci curvature} on a
Riemannian manifold has significant consequences for the associated
heat semigroup/Brownian motion. In fact, such a bound implies a
logarithmic Sobolev inequality, a Poincar\'e inequality, and a
Brunn--Minkowski inequality, as well as several geometric
inequalities.

Because of this wide range of implications, major research activity
has been devoted to developing a notion of Ricci curvature (lower
boundedness) that applies to non-smooth settings. Several approaches
have been developed. Bakry--Em\'ery \cite{BE85} introduced an approach
based on algebraic properties of diffusion operators (the so-called
$\Gamma_2$-calculus). Later, an approach based on optimal transport
has been developed by Lott, Sturm and Villani \cite{S06,LV09}, and
subsequently refined by Ambrosio, Gigli and Savar\'e \cite{AGS11b}. In
recent years, the equivalence of the algebraic approach and the
optimal transport approach has been proved, and a complete picture is
emerging.
 
However, since the theory does not apply to discrete settings, several
discrete notions of Ricci curvature have been introduced. In
particular, the notion of \emph{coarse Ricci curvature} was developed
in considerable detail by Ollivier \cite{Oll09}, although the basics
were implicit in Dobrushin's work and others' since (see, e.g., the
discussion in \cite{Oll10}), besides the notion being made explicit in
the Ph.D. thesis of Sammer \cite{Sam05}. This notion is based on
contraction properties of a Markov kernel in the (Kantorovich)
$W_1$-metric. In this paper we focus on a different notion of Ricci
curvature, which was proposed in \cite{Ma11} and systematically
studied in \cite{EM12}.

\subsection{A discrete notion of Ricci curvature}

Let $L$ be the generator of a continuous time Markov chain on a finite
set $\cX$, thus for functions $\psi : \cX \to \R$, the operator $L$ is
of the form $L \psi(x) = \sum_{y \in \cX} Q(x,y) (\psi(y) - \psi(x))$
where $Q(x,y) \geq 0$ for all $x, y \in \cX$ with $x \neq y$, and $Q(x,x) = 0$ for all $x \in \cX$. We shall
assume that there exists a reversible probability measure $\pi$ on
$\cX$, which means that $\pi(x) Q(x,y) = \pi(y)Q(y,x)$ for all $x,
y$. We let $\PX=\{\rho\in\R_+^{\cX}:\sum_x\rho(x)\pi(x)=1\}$ be the
space of probability densities on $\cX$ and denote by
$\cH(\rho)=\sum_x\rho(x)\log\rho(x)\pi(x)$ the relative entropy of
$\rho\in\PX$.

In \cite{Ma11} a metric $\cW$ on the space of probability measures
has been constructed with the property that the heat flow is the
gradient flow of the relative entropy. In this sense, $\cW$ may be
regarded as a natural analogue of the $2$-Wasserstein metric induced
by the Markov triple $(\cX, Q, \pi)$. We refer to Section
\ref{sec:prelim} for the precise definition of $\cW$.

We say that $(\cX, Q, \pi)$ has \emph{Ricci curvature bounded from
  below by $\kappa \in \R$} if the relative entropy $\cH$ is
$\kappa$-geodesically convex along $\cW$-geodesics. More explicitly,
for any constant speed geodesic $\{\rho_t\}_{t \in [0,1]}$ in
$(\PX,\cW)$, we require that
  \begin{align*}
    \cH(\rho_t) \leq (1-t) \cH(\rho_0) + t \cH(\rho_1) -
    \frac\kappa{2} t(1-t) \cW(\rho_0, \rho_1)^2\;.
  \end{align*}
  In this case, we write
  \begin{align*}
    \Ric(\cX, Q, \pi) \geq \kappa\;.
  \end{align*}
  This notion of Ricci curvature is a direct analogue of the notion
  introduced by Lott, Sturm, and Villani in the setting of geodesic
  metric measure spaces.

It has been shown in \cite{EM12} that this notion of Ricci curvature
has significant consequences, such as an HWI-inequality \`a la Otto--Villani, a modified logarithmic Sobolev
inequality (MLSI) and a Poincar\'e (or spectral gap) inequality. The MLSI (with constant $\alpha > 0$) asserts that 
\begin{align*}
  \cH(\rho)\leq{\alpha}^{-1}\cE(\rho,\log\rho)\;,
\end{align*}
for all $\rho \in \PX$, where $\cE$ is the associated Dirichlet form defined by 
\begin{align*}
  \mathcal{E}(f,g)~=~ - \ip{f, Lg}_{L^2(\cX,\pi)} ~=~\frac12\sum_{x,y \in \cX}\big(f(y)-f(x)\big)\big(g(y)-g(x)\big)Q(x,y)\pi(x)\;.
\end{align*}
The MLSI is equivalent to the exponential convergence estimate $\cH(e^{tL}\rho) \leq e^{-\alpha t} \cH(\rho)$. The Poincar\'e inequality asserts that
\begin{align*}
 \| \psi \|_{L^2(\cX,\pi)}^2 \leq \lambda^{-1} \cE(\psi, \psi)\,,
\end{align*}
for all functions $\psi : \cX \to \R$ with $\sum_{x \in \cX}
\psi(x)\pi(x) = 0$. It is equivalent to the exponential convergence
estimate $\|e^{t L} \psi\|_{L^2(\cX, \pi)} \leq e^{- \lambda t} \|
\psi\|_{L^2(\cX, \pi)}$. From now on, we will denote the optimal
constants in the inequalities by $\kappa$, $\alpha$ and $\lambda$
respectively. It is well known that $\lambda \geq \alpha /
2$. Moreover, it has been proved in \cite{EM12} that $\alpha/2 \geq
\kappa$.

In view of these consequences it is desirable to obtain sharp Ricci
curvature bounds in concrete discrete examples. So far, very little is
known in this direction. Two types of results have been obtained:
\begin{itemize}
\item Mielke \cite{Mie11b} obtained Ricci curvature bounds for
  one-dimensional birth-death chains. He applies his bounds to
  approximations of Fokker--Planck equations with $\kappa$-convex
  potential and shows that the curvature of the discrete
  approximations converge to $\kappa$. The proof relies on diagonal
  dominance of the Hessian matrix, and seems to be restricted to
  1-dimensional situations.
\item Erbar--Maas \cite{EM12} obtained a tensorisation result for
  Ricci curvature: if $\Ric(\cX_i, Q_i, \pi_i) \geq \kappa_i$ for $i =
  1, 2$, then the associated product chain on the product space $\cX_1 \times
  \cX_2$ has Ricci curvature bounded from below by $\min\{\kappa_1,
  \kappa_2\}$.
\end{itemize}
Apart from the elementary example of the
complete graph, no results are available beyond the 1-dimensional or
the product setting. This paper provides the first results in this
direction.

In a different direction, Gozlan et al \cite{GRST13} constructed an
interpolation on the space of probability measures and derived a
displacement convexity of entropy inequality with respect to the
classical $W_1$-metric on the complete graph and products of complete
graphs, in particular, the $n$-dimensional discrete cube; the results
thus obtained are consistent with the bounds on the curvature in the
sense \cite{EM12} as well as the coarse Ricci curvature.

\subsection{The Bernoulli-Laplace model}

The Bernoulli-Laplace model is the simple exclusion process
on the complete graph and can be described as follows. Consider $k$
indistinguishable particles distributed over $n$ sites labeled by $[n]
= \{1, \ldots, n\}$, where $1 \leq k < n$. Each site contains at most
one particle. The state space of the system is the set $\Omega(n,k)=
\{x \in\{0,1\}^n \ : \ x_1 + \cdots + x_n = k \}$ (or equivalently,
the set of all subsets of $[n]$ of size $k$).

The Bernoulli-Laplace model is the continuous time Markov chain
described as follows: after random waiting times (independent
exponentially distributed with rate $\frac{1}{k(n-k)}$), one particle
is selected uniformly at random, and jumps to a free site, selected
uniformly at random. The transition rates are thus given by
\begin{align*}
 Q_{\rm BL}(x,y) =  \left\{ \begin{array}{ll}
\frac{1}{k(n-k)}\;,
 & \text{if $\| x-y\|_{\ell^1} = 2$ },\\
0\;,
 & \text{otherwise}\;.\end{array} \right.
\end{align*}
The uniform probability measure on $\Omega(n,k)$, given by $\pi_{\rm
  BL}(x) = \binom{n}{k}^{-1}$ for all $x$, is reversible for $Q_{\rm
  BL}$. Note that the Bernoulli-Laplace model may be seen as the
simple random walk on $\Omega(n,k)$ endowed with the Hamming distance
$d(x, y) = \frac{1}{2} \| x-y\|_{\ell^1}$.

We prove the following result:

\begin{theorem}[Ricci bound for the Bernoulli-Laplace
  model]\label{thm:BL-main}
  Let $n > 1$ and $1 \leq k \leq n-1$. The Ricci curvature of the
  Bernoulli-Laplace model $(\Omega(n,k),Q_{\rm BL},\pi_{\rm BL})$ is
  bounded from below by $\frac{n+2}{2k(n-k)}$.
\end{theorem}

The mixing time for the Bernoulli-Laplace model has been studied by
Diaconis and Shashahani \cite{DiSh87}, who showed in particular that
the spectral gap equals $\frac{n}{k(n-k)}$. Their analysis is based on
lifting the model to the symmetric group and using representation
theory in this setting.  Lee and Yau \cite{LY98} obtained a sharp
logarithmic Sobolev inequality, improving earlier work by Diaconis and
Saloff-Coste \cite{DSC96}.  In three independent works Gao--Quastel
\cite{GQ03}, Goel \cite{Goe04}, and Bobkov--Tetali \cite{BT06} proved
the following (lower) bound on the MLSI constant:
\begin{align*}
  \frac{n}{2k(n-k)}~\leq~\a~\leq~\frac{2n}{k(n-k)}\;,
\end{align*}
where the upper bound comes from the fact that $\a\leq2\lambda$.
Since $\alpha\geq 2\kappa$ by \cite[Thm. 7.4]{EM12}, our Theorem
\ref{thm:BL-main} implies that
\begin{align*}
 \frac{n+2}{k(n-k)}~\leq~\a~\leq~\frac{2n}{k(n-k)}\;,
\end{align*}
which improves the lower bound above roughly by a factor $2$. Such an
improvement on the MLSI constant for the Bernoulli-Laplace model has
previously been obtained by Caputo et al. \cite{CDPP09}.

\subsection{The random transposition model}

Let $n \geq 1$, and let $\cS_n$ be the group of all permutations of
$[n]$. We define a graph structure on $\cS_n$ by connecting two
permutations $\sigma_1, \sigma_2 \in \cS_n$ if $\sigma_2 = \tau \circ
\sigma_1$, for some transposition $\tau$. (Recall that a transposition
is a permutation that interchanges precisely two elements). In this
case we write $\sigma_1 \sim \sigma_2$. Simple random walk is then
defined by
\begin{align*}
 Q_{\rm RT}(\sigma_1,\sigma_2) =  \left\{ \begin{array}{ll}
\frac{2}{n(n-1)}\;,
 & \text{if $\sigma_1 \sim \sigma_2$ },\\
0\;,
 & \text{otherwise}\;.\end{array} \right.
\end{align*}
The uniform measure $\pi_{RT}$ given  $\pi_{RT}(\sigma) = 1 / n!$ is reversible for $Q_{\rm RT}$.

\begin{theorem}[Ricci bound for the random transposition
  model]\label{thm:RT-main}
Let $n > 1$. The Ricci curvature of the random transposition model $(\cS_n,Q_{\rm RT},\pi_{\rm RT})$ is bounded from below by $\frac{4}{n(n-1)}$.
\end{theorem}

As mentioned above the mixing time for $\cS_n$ has been obtained by
Diaconis and Shahshahani in \cite{DiSh81}.  The {\em coarse Ricci
  curvature} of the random transposition model can be estimated from
above and below in a straightforward manner using contraction of the
$W_1$-transportation distance (as observed by Gozlan et al
\cite{GMPRST13}, while very likely in the folklore) and shown to be of
order $n^{-2}$. The modified logarithmic Sobolev inequality was
studied by Goel \cite{Goe04}, Gao--Quastel \cite{GQ03} and
Bobkov--Tetali \cite{BT06}, who proved that
\begin{align*}
  \frac{1}{n-1}~\leq~\a~\leq~\frac{4}{n-1}\;,
\end{align*}
where the upper bound comes from the known spectral gap
$\lambda=\frac2{n-1}$. Thus $\a$ and $\lambda$ are both of order
$n^{-1}$. Combining this estimate with Theorem \ref{thm:RT-main}, we
infer that $4/(n^2-n)\leq\kappa\leq 2/(n-1)$. It remains an open
question to determine the correct order.

\section{Preliminaries on Ricci curvature}
\label{sec:prelim}

We briefly recall some preliminaries on the notion of Ricci curvature
for finite Markov chains following \cite{Ma11,EM12,EM13}.

\subsection{Ricci curvature for Markov triples}
\label{sec:def}

Let $L$ be the generator of a continuous time Markov chain on a finite
set $\cX$. Thus the action of $L$ on functions $\psi : \cX \to \R$ is
given by
\begin{align*}
 L \psi(x) := \sum_{y \in \cX} Q(x,y) (\psi(y) - \psi(x))\;, \qquad x \in \cX\;,
\end{align*}
with $Q(x,y) \geq 0$ for all $x \neq y$.  Let $\pi$ be a reversible
measure for $L$, i.e. the detailed balance conditions
\begin{align*}
 Q(x,y) \pi(x) = Q(y,x) \pi(y)
\end{align*}
hold for all $x \neq y$. We refer to the triple $(\cX, Q, \pi)$ as a  Markov triple.
 
Let 
\begin{align*}
 \PX := \Big\{ \, \rho : \cX \to \R_+ \ | \ 
     \sum_{x \in \cX} \pi(x) \rho(x)  = 1 \, \Big\}
\end{align*}
be the set of \emph{probability densities} (with respect to $\pi$) on
$\cX$. The subset consisting of those probability densities that are
strictly positive is denoted by $\PXs$.
 
A crucial role in this paper is played by the nonlocal transport
metric $\cW$ on $\PX$, which was introduced in \cite{Ma11,Mie11a} (see
also \cite{CHLZ11} for closely related metrics). In several ways, this
metric can be regarded as a natural discrete analogue of the
2-Wasserstein metric \cite{GM12}. The definition is based on a
discrete analogue of the Benamou-Brenier formula: for $\rho_0, \rho_1
\in \PX$ we set
\begin{align*}
 \cW(\rho_0, \rho_1)^2
   := \inf_{\rho, \psi} 
   \bigg\{  \frac12   \int_0^1 
  \sum_{x,y\in \cX} (\psi_t(x) - \psi_t(y))^2
    		 \hat\rho_t(x,y)  Q(x,y)\pi(x)
      \dd t 
          \bigg\}\;,
\end{align*}
where the infimum runs over all piecewise smooth curves $\rho :
[0,1] \to \PX$ and all $\psi : [0,1] \times \cX \to \R$ satisfying the
discrete ``continuity equation''
\begin{align} \label{eq:cont}
 \begin{cases}
 \displaystyle\ddt \rho_t(x) 
   + \displaystyle\sum_{y \in \cX} ( \psi_t(y) - \psi_t(x) ) \hat\rho_t(x,y) Q(x,y) ~=~0\;,\qquad x \in \cX\;, \\ 
  \rho(0) = \rho_0\;, \qquad \rho(1)  = \rho_1\;.
 \end{cases}
\end{align}
Here, given $\rho \in \PX$, we write $\hat\rho(x,y) :=
\theta\big(\rho(x),\rho(y)\big)$, where $\theta(r,s) = \int_0^1
r^{1-p} s^p \dd p$ is the logarithmic mean of $r$ and $s$.

The relative entropy (with respect to $\pi$) of $\rho \in \PX$ is
defined as usual by
\begin{align} \label{eq:entropy}
 \cH(\rho) = \sum_{x \in \cX} \pi(x) \rho(x) \log \rho(x)\;.
\end{align}
It turns out that the metric $\cW$ is induced by a Riemannian
structure on the interior $\PXs$ of $\PX$. Moreover, every pair of
densities $\rho_0, \rho_1 \in \PX$ can be joined by a constant speed
geodesic, i.e., there exists a curve $\rho : [0,1] \to \PX$ connecting
$\rho_0$ and $\rho_1$ satisfying $\cW(\rho_s, \rho_t) = |t-s|
\cW(\rho_0, \rho_1)$ for all $s, t \in [0,1]$. Therefore, the
following definition in the spirit of Lott--Sturm--Villani
\cite{LV09,S06} is meaningful.

\begin{definition}[Discrete Ricci curvature]\label{def:intro-Ricci}
  We say that a Markov triple $(\cX, Q, \pi)$ has \emph{Ricci
    curvature bounded from below by $\kappa \in \R$} if for any
  constant speed geodesic $\{\rho_t\}_{t \in [0,1]}$ in $(\PX, \cW)$
  we have
  \begin{align*}
    \cH(\rho_t) \leq (1-t) \cH(\rho_0) + t \cH(\rho_1) -
    \frac\kappa{2} t(1-t) \cW(\rho_0, \rho_1)^2\;.
  \end{align*}
  In this case, we write $\Ric(\cX, Q, \pi) \geq \kappa$, or simply $\Ric(Q) \geq \kappa$.
\end{definition}

\subsection{Equivalent conditions for Ricci curvature}
\label{sec:BA}

To proceed further, we introduce the following convenient notation.
For a function $\phi: \cX \to \R$ we consider the \emph{discrete
  gradient} $\nabla \phi \in \R^{\cX \times \cX}$ defined by
\begin{align*}
 \nabla \phi(x,y) := \phi(y) - \phi(x)\;.
\end{align*}
For $\Psi \in \R^{\cX \times \cX}$ we consider the \emph{discrete
  divergence} $\nabla \cdot \Psi \in \R^\cX$ defined by
\begin{align*}
( \nabla \cdot \Psi )(x) 
  := \frac12 \sum_{y \in \cX}  (\Psi(x,y) - \Psi(y,x) ) Q(x,y) \in \R\;.
\end{align*}
With this notation we have
$L :=  \nabla \cdot \nabla$, and the integration by parts formula
\begin{align*}
 \ip{\nabla \psi, \Psi}_\pi = -\ip{\psi,\nabla\cdot \Psi}_\pi
\end{align*}
holds. Here we write, for $\phi,\psi \in \R^\cX$ and $\Phi,\Psi \in \R^{\cX \times \cX}$,
\begin{align*}
\ip{\phi, \psi}_\pi &= \sum_{x \in \cX} \phi(x) \psi(x) \pi(x)\;, \\
\ip{\Phi, \Psi}_\pi &= \frac12 \sum_{x,y \in \cX} 
         \Phi(x,y) \Psi(x,y) Q(x,y) \pi(x)\;.
\end{align*}

An important role in our analysis is played by the quantity
$\cB(\rho, \psi)$, which is defined for $\rho \in \R_+^\cX$ and $\psi \in
\R^\cX$ by
\begin{equation}\begin{aligned}\label{eq:B}
  \cB(\rho, \psi) 
  := &\ \frac12
  \bip{\hL\rho\, \cdot\, \nabla \psi, \nabla \psi }_\pi
    - \bip{\hrho  \cdot\,\nabla \psi\ ,\, \nabla L\psi}_\pi 
\\  = &\ 
  \frac14
       \sum_{x,y,z \in \cX}
         \big(\psi(x) - \psi(y)\big)^2 
        \Big( \hrho_1(x,y) \big(\rho(z) - \rho(x) \big)Q(x,z)
 \\& \qquad
    + \hrho_2(x,y) \big(\rho(z) - \rho(y) \big)Q(y,z)\Big)
       Q(x,y) \pi(x)
  \\ &  - 
   \frac12
       \sum_{x,y,z \in \cX}
         \Big( Q(x,z) \big( \psi(z) - \psi(x) \big)
           - Q(y,z)   \big( \psi(z) -\psi(y)\big)
             \Big)
  \\& \qquad \times       \big(\psi(x) - \psi(y)\big)
        \hrho(x,y)Q(x,y)\pi(x)\;,
\end{aligned}\end{equation}
where
\begin{align*}
\hat\rho(x,y) &:= \theta(\rho(x),\rho(y))\;,\\
\hrho_i(x,y) &:= \partial_i\theta(\rho(x),\rho(y))\;,\quad i=1,2\;,\\
 \hL \rho(x,y) &:=  
 \hrho_1(x,y) L \rho(x) 
  + \hrho_2(x,y) L \rho(y)\;.
\end{align*}
The term $\cB(\rho, \psi)$ is reminiscent of the Bochner formula in
Riemannian geometry, which asserts that $\frac12\Delta|\nabla \psi|^2
- \ip{\nabla \Delta \psi, \nabla \psi} = \Ric(\nabla \psi, \nabla
\psi) + \|D^2 \psi\|_{HS}^2$.

Let us further introduce the quantity
\begin{align*}
 \cA(\rho,\psi) := 
  \ip{\hat\rho\cdot\nabla \psi,\nabla\psi}
  = \frac12\sum_{x,y\in \cX} (\psi(y)-\psi(x))^2\hrho(x,y) Q(x,y) \pi(x)\ ,
\end{align*}
for $\rho\in\PX$ and $\psi \in \R^{\cX}$.

The following result from \cite{EM12} provides a
reformulation of Ricci lower bounds in terms of $\cB$ and $\cA$.

\begin{theorem}[Characterisation of Ricci curvature bounds]\label{thm:Ric-equiv}
  Let $\kappa \in \R$. For an irreducible and reversible Markov kernel
  $(\cX, Q, \pi)$ the following assertions are equivalent:
\begin{enumerate}
\item $\Ric(\cX, Q, \pi) \geq \kappa$\;;
\item For all $\rho \in \PXs$ we have 
\begin{align*}
\Hess \cH(\rho) \geq \kappa\;;
\end{align*}
\item For all $\rho \in \PXs$ and $\psi \in \R^\cX$ we have
\begin{align*}
 \cB(\rho, \psi) \geq \kappa \cA(\rho, \psi)\;.
\end{align*}
\end{enumerate}
\end{theorem}

The equivalence between (1) and (2) shows equivalence of a non-smooth
and a smooth notion of convexity. This equivalence is non-trivial,
since the Riemannian metric is degenerate at the boundary. Assertion
(3) is an explicit reformulation of (2). The inequality in (3) can be
seen as a discrete analogue of Bochner's inequality.

\section{A simple criterion for Ricci curvature bounds}
\label{sec:crit}

Here we present a combinatorial method for controlling the quantity
$\cB$ from \eqref{eq:B}. We will first study this quantity in detail
in the case where the Markov chain is simple random walk on a triangle
or on a square. The resulting bounds will then be applied to concrete
examples with sufficient symmetry in which the underlying graph can be
decomposed into squares and triangles.

\subsection{Decomposition of $\cB(\rho, \psi)$}
Let us consider the natural graph structure $(\cX,E)$ associated with
the kernel $Q$, where the set of edges is defined by
\begin{align*}
  E~:=\big\{\{x,y\}\ :\ Q(x,y)>0\big\}\;.
\end{align*}
Then we can rewrite the quantity $\cA$ as
\begin{align*}
  \cA(\rho,\psi)~=~\sum\limits_{e\in E}a(e)c(e)\;,
\end{align*}
where for $e=\{x,y\}$ we set $c(e)=Q(x,y)\pi(x)$ and 
\begin{align*}
  a(e)~=~\big(\psi(y)-\psi(x)\big)^2\hrho(x,y)\;.
\end{align*}
Given two edges $e,e'\in E$ we write $e\sim e'$ iff they are adjacent
or identical, i.e., iff $e=\{x,y\},e'=\{x,z\}$ for some $x,y,z\in
\cX$.  Then we can rewrite the quantity $\cB$ as a sum over pairs of
adjacent edges. It will be convenient to write
$c(x,y):=Q(x,y)\pi(x)$. Note that the reversibility assumption implies
that $c(x,y) = c(y,x)$.

\begin{lemma}[Reformulation of $\cB(\rho, \psi)$]\label{lem:edgesum1}
  For all $\rho\in\R_+^\cX$ and $\psi\in\R^{\cX}$ we have
  \begin{align}\label{eq:edgesum1}
    \cB(\rho,\psi)~=~\sum\limits_{e,e'\in E, e\sim e'}b(e,e')\;,
  \end{align}
  where for $e = \{x,y\}$ and $e' = \{x,z\}$ with $y\neq z$ we set
  \begin{align*}
    b(e,e')~&:=~
      \Big[\frac12\big(\psi(x) - \psi(y)\big)^2\hat\rho_1(x,y)\big(\rho(z) - \rho(x) \big)  \\
      &\qquad + \big( \psi(y) - \psi(x) \big) \big( \psi(z)
      -\psi(x)\big)\hrho(x,y)\Big]Q(x,z)c(x,y)\;,
  \end{align*}
   while for $y=z$ we set
  \begin{align*}
    b(e,e)~&:=~
      \frac12 \big(\psi(x) - \psi(y)\big)^2
          \bigg[ 
          2\hrho(x,y)\big[Q(x,y)+Q(y,x)\big]
\\&   \qquad\qquad\qquad           + \hrho_1(x,y) \big(\rho(y) - \rho(x) \big)Q(x,y)
                + \hrho_2(x,y) \big(\rho(x) - \rho(y) \big)Q(y,x)
   \bigg]c(x,y)\;.
  \end{align*}
\end{lemma}

\begin{proof}
  First note that using the fact that $\hat\rho(x,y)=\hat\rho(y,x)$,
  $\hat\rho_1(x,y)=\hat\rho_2(y,x)$ and the detailed balance condition
  $Q(x,y)\pi(x)=Q(y,x)\pi(y)$ we can rewrite \eqref{eq:B} in the form
  \begin{align*}
    \cB(\rho,\psi)~&=~\frac12
     \sum_{x,y,z \in \cX}
         \big(\psi(x) - \psi(y)\big)^2 
       \hat\rho_1(x,y)\big(\rho(z) - \rho(x) \big)Q(x,z)Q(x,y)\pi(x)\\
&\qquad +
     \sum_{x,y,z \in \cX}
         \big( \psi(y) - \psi(x) \big)( \psi(z) -\psi(x)\big)
        \hrho(x,y)Q(x,z)Q(x,y)\pi(x)\;.
  \end{align*}
  Now the assertion is obvious.
\end{proof}

Given a subgraph $G=(\cY,F)$ of $(\cX,E)$ and two functions $\rho\in
\R_+^\cX$ and $\psi\in\R^{\cX}$ we denote by $\rho^G,\psi^G$ their
restrictions to $\cY$. Moreover, we set $\cA_G(\rho,\psi)$ and
$\cB_G(\rho,\psi)$ to be the quantities $\cA,\cB$ calculated in the
weighted graph $(\cY,F)$ with the functions $\rho^G,\psi^G$. More
precisely,
\begin{align*}
  \cA_G(\rho,\psi)~&:=~\sum\limits_{e\in F}a(e)c(e)\;, \\
  \cB_G(\rho,\psi)~&:=~\sum\limits_{e,e'\in F, e\sim e'}b(e,e')\;.
\end{align*}
Further, it turns out to be useful to seperate the contribution to
$\cB$ coming from identical edges (``on-diagonal entries'') and from
adjacent edges (``off-diagonal entries''). Thus we set
\begin{align*}
  \cB^{\rm on}_G(\rho,\psi)~&:=~\sum\limits_{e\in F}b(e,e)\;,\\
  \cB^{\rm off}_G(\rho,\psi)~&:=~\sum\limits_{e,e'\in F, e\sim e',e\neq e'}b(e,e')\;.
\end{align*}

\subsection{An on-diagonal bound for $d$-regular graphs}

From now on let us assume that the Markov chain is simple random walk
on a $d$-regular graph $(\cX,E)$, i.e., 
\begin{align*}
  Q(x,y)~=~
  \begin{cases}
    \frac{1}{d}\;,& \{x,y\}\in E\;,\\
    0\;,& \text{otherwise}\;.
  \end{cases}
\end{align*}
The uniform probability measure given by $\pi(x) = \mu = |\cX|^{-1}$
for all $x \in \cX$ satisfies the detailed balance condition.

The following results is a general bound on the on-diagonal part of
$\cB$. In the proof we shall use the following elementary properties
of the logarithmic mean:
 \begin{align}\label{eq:theta1}
   s\partial_1\theta(s,t) + t\partial_2(s,t)~&=~\theta(s,t)\;,\\\label{eq:theta2}
   u\partial_1\theta(s,t) + v\partial_2(s,t)~&\geq~\theta(u,v)\;,
 \end{align}
for all $s,t,u,v>0$. A proof can be found in \cite[Lemma 2.2]{EM12}.

\begin{lemma}[On-diagonal bound]\label{lem:diag}
  For every subgraph $G\subset (\cX,E)$ and all $\rho\in\R_+^\cX$ and
  $\psi\in\R^\cX$ we have
  \begin{align}\label{eq:diag}
    \cB^{\rm on}_G(\rho,\psi)~\geq~\frac2d \cA_G(\rho,\psi)\;.
  \end{align}
\end{lemma}

\begin{proof}
  Let us write $G=(Y,F)$. Using \eqref{eq:theta1} and \eqref{eq:theta2} we obtain
  \begin{align*}
    &\cB^{\rm on}_G(\rho,\psi)~\\
    ~&=~\frac{\mu}{d^2}\sum\limits_{\{x,y\}\in F}\frac12\Big(\psi(y)-\psi(x)\Big)^2\Big[\hrho_1(x,y)\big(\rho(y)-\rho(x)\big)
    + \hrho_2(x,y)\big(\rho(x)-\rho(y)\big) +4\hrho(x,y) \Big]\\
     &\geq~\frac{2\mu}{d^2 }\sum\limits_{\{x,y\}\in F}\Big(\psi(y)-\psi(x)\Big)^2\hrho(x,y)~=~\frac{2}{d}\cA_G(\rho,\psi)\;,
  \end{align*}
which is the desired bound.
\end{proof}

\subsection{Off-diagonal bounds for triangles and squares}

For the off-diagonal part, let us first consider the special cases
where the subgraph $G$ is a triangle or a square.

\begin{lemma}[Off-diagonal bound for triangles]\label{lem:triangle}
  Let $\triangle=(\cY,F)$ be a triangle subgraph of $(\cX,E)$,
  i.e., $\cY=\{x_1,x_2,x_3\}$ and
  $F=\big\{\{x_i,x_{i+1}\},~i=1,2,3\big\}$ for some distinct $x_i\in\cX$. Then,
for any $\rho\in\R_+^\cX$ and $\psi\in\R^\cX$   we have 
  \begin{align}\label{eq:trianle}
    \cB_\triangle^{\rm off}(\rho,\psi)~&\geq~
\frac{1}{2d} \cA_\triangle(\rho,\psi)\;.
  \end{align}
\end{lemma}

\begin{proof}
  For convenience we set $\rho_i=\rho(x_i)$ and
  $g_i=\psi(x_{i+1})-\psi(x_i)$ for $i=1,2,3$ with the convention that
  $x_0 = x_3$ and $x_4=x_1$. To simplify notation we write
  $\hrho_{i,j} = \hrho(x_i, x_j)$, $\hrho^1_{i,j} = \hrho_1(x_i, x_j)$
  and $\hrho^2_{i,j} = \hrho_2(x_i, x_j)$. It is readily verified that
  \begin{align*}
   \cB^{\rm off}_\triangle(\rho,\psi)~&=~
    \frac{\mu}{d^2}\sum\limits_{i=1}^3
    \frac12 g_i^2 \Big[\hrho^1_{i,i+1}(\rho_{i-1}-\rho_i) + \hrho^2_{i,i+1}(\rho_{i+2}-\rho_{i+1})\Big]
    -g_i(g_{i+1}+g_{i-1})\hrho_{i,i+1}\;.
  \end{align*}
  Using the inequality $\hrho^1_{ij} \geq 0$, the identity
  \eqref{eq:theta1}, and the fact that $g_1+g_2+g_3=0$, we estimate
    \begin{align*}
   \cB^{\rm off}_\triangle(\rho,\psi)~&\geq~
    \frac{\mu}{d^2}\sum\limits_{i=1}^3
    -\frac12 g_i^2 \hrho_{i,i+1}
    -g_i(g_{i+1}+g_{i-1})\hrho_{i,i+1}
\\&=   \frac{\mu}{2d^2}\sum\limits_{i=1}^3 g_i^2 \hrho_{i,i+1}~=~\frac1{2d}\cA_\triangle(\rho,\psi)\;,
  \end{align*}
which completes the proof.
\end{proof}

For $s,t,u,v > 0$ let $\sD(s,t;u,v) := u\partial_1\theta(s,t) +
v\partial_2(s,t) - \theta(u,v)$ be the deficit in the 4-point
inequality \eqref{eq:theta2}, thus $\sD(s,t;u,v) \geq 0$. To simplify
notation we will often write $\sD_{i,j}^{k,l}$ instead of
$\sD(\rho_i,\rho_j;\rho_k,\rho_l)$. The following result provides a
convenient representation of $\cB_\Box^{\rm off}$ as a sum of
nonnegative terms.

\begin{lemma}[Off-diagonal bound for squares]\label{lem:square}
  Let $\Box=(\cY,F)$ be a square subgraph of $(\cX,E)$,
  i.e., $\cY=\{x_1,x_2,x_3,x_4\}$ and
  $F=\big\{\{x_i,x_{i+1}\},~i=1,\cdots,4\big\}$ for some distinct $x_i\in\cX$. Then,
 for any $\rho\in\R_+^\cX$ and $\psi\in\R^\cX$ we have
  \begin{align*}
    \cB_\Box^{\rm off}(\rho,\psi)
     = \frac{\mu}{2 d^2}|AS|^2(\psi;\Box) \sum_{i = 1}^4 \hrho(x_i,x_{i+1}) 
       + \frac{\mu}{2d^2} \sum_{i=1}^4 \big(\psi(x_{i+1})-\psi(x_i)\big)^2 \sD_{i,i+1}^{i-1,i+2}
    ~\geq~0\;,
  \end{align*}
  where $|AS|(\psi;\Box) := |\psi(x_1) - \psi(x_2) + \psi(x_3) -
  \psi(x_4)|$ denotes the alternating sum of $\psi$ on $\Box$.
\end{lemma}

Note that the definition of $|AS|(\psi;\Box)$ does not depend on the
parametrisation of $\Box$.

\begin{proof}
  We set $\rho_i=\rho(x_i)$ and $g_i=\psi(x_{i+1})-\psi(x_i)$ for
  $i=1,2,3,4$ with the convention $x_0 = x_4$ and $x_5=x_1$. Moreover,
  we define $\hrho_{i,j}$ and $\hrho^1_{i,j},\hrho^2_{i,j}$ as in the
  proof of Lemma \ref{lem:triangle}. Using \eqref{eq:theta1},
  \eqref{eq:theta2} and the identity $g_1+g_2+g_3+g_4=0$ we obtain
  \begin{align*}
    \cB^{\rm off}_\Box(\rho,\psi)
  ~&=~ \frac{\mu}{d^2}\sum\limits_{i=1}^4
    \frac12 g_i^2 \Big[\hrho^1_{i,i+1}(\rho_{i-1}-\rho_i) + \hrho^2_{i,i+1}(\rho_{i+2}-\rho_{i+1})\Big]
 -g_i(g_{i+1}+g_{i-1})\hrho_{i,i+1}\\
    ~&=~ \frac{\mu}{d^2}\sum\limits_{i=1}^4
    \frac12 g_i^2 \Big[\hrho^1_{i,i+1} \rho_{i-1}  + \hrho^2_{i,i+1} \rho_{i+2}\Big]
 -g_i(g_{i+1}+\frac12 g_i+g_{i-1})\hrho_{i,i+1}\\
  &=~\frac{\mu}{d^2}\sum\limits_{i=1}^4
    \frac12 g_i^2 \Big[ \hrho_{i-1,i+2} + \sD_{i,i+1}^{i-1,i+2} \Big]
   + g_i(\frac12 g_i + g_{i+2})\hrho_{i,i+1}\\
    &=~\frac{\mu}{4d^2}\sum\limits_{i=1}^4  (g_i+g_{i+2})^2
    \Big[\hrho_{i-1,i+2} + \hrho_{i,i+1}\Big]
    + 2 g_i^2 \sD_{i,i+1}^{i-1,i+2}
    \\
    &=~\frac{\mu}{8d^2}\sum\limits_{i=1}^4  (g_i+g_{i+2})^2
   \Big[\sum\limits_{j=1}^4 \hrho_{j,j+1}\Big]
   + 4 g_i^2 \sD_{i,i+1}^{i-1,i+2}\;,
 \end{align*}
which yields the desired identity.
\end{proof}

\begin{remark}\label{rem:square-sharp}
The bound $\cB^{\rm off}_\Box(\rho,\psi) \geq 0$ is sharp, in the sense that there exist $\rho\in\R_+^\cX$ and $\psi\in\R^\cX$ with $\cB^{\rm off}_\Box(\rho,\psi) = 0$ and $\cA(\rho, \psi) > 0$. Take for instance $\rho_i = 1$ for all $i$ and a non-nonstant function $\psi$ with $|AS|(\psi;\Box) = 0$.
\end{remark}

\section{The Bernoulli--Laplace model}
\label{sec:bernlap}

For integers $n>1$ and $1\leq k\leq n-1$ consider the $k$-slice of the
$n$-dimensional discrete cube
\begin{align*}
  \Omega(n,k) = \big\{x\in\{0,1\}^n~:~x_1+\cdots + x_n=k \big\}\;.
\end{align*}
Two points in $\Omega(n,k)$ are declared neighbors if they differ in exactly two coordinates. Let us set
\begin{align*}
  I(x) = \{i\leq n~:~x_i=1\}\;,\quad  J(x) = \{j\leq n~:~x_j=0\}\;.
\end{align*}
Then the neighbors of $x$ are given by $\{s_{ij}x\}_{i\in I(x),j\in J(x)}$, where
\begin{align*}
 (s_{ij}x)_i=0\;,~(s_{ij}x)_j=1\;,(s_{ij}x)_k=x_k\quad\forall k\neq i,j
\end{align*}
Note that every point $x\in\Omega(n,k)$ has $k(n-k)$
neighbors, and that the set of edges is $E=\big\{\{x,s_{ij}x\}~:~ x \in \Omega(n,k),\; i\in I(x),\;j\in J(x)\big\}$. The simple random walk on $(\Omega(n,k),E)$ is given by $Q_{\rm BL}(x,y)=(k(n-k))^{-1}$ whenever $x\sim y$, and has as
invariant measure the uniform measure $\pi(x)=\abs{\Omega(n,k)}^{-1}=\binom{n}{k}^{-1}$.

We have the following curvature bound for the Bernoulli Laplace model (Theorem \ref{thm:BL-main}).

\begin{theorem}\label{thm:bernlap}
  The simple random walk $Q_{\rm BL}$ on $\Omega(n,k)$ satisfies
  \begin{align*}
     \Ric(Q_{\rm BL}) \geq \frac{n+2}{2k(n-k)}\;.  
  \end{align*}
\end{theorem}

\begin{proof}
  Let us set $d=k(n-k)$ and $\mu = \binom{n}{k}^{-1}$. Then we need to show that for any
  $\rho\in \cP_*(\Omega(n,k))$ and any $\psi : \Omega(n,k) \to \R$ we have
  \begin{align*}
    \cB(\rho,\psi)~\geq~\frac{n+2}{2d}\cA(\rho,\psi)\;.
  \end{align*}
  Let $P=\big\{(e,e')\in E\times E~:~e\sim e', e\neq e'\big\}$ be the
  set of pairs of adjacent non-identical edges. We define a
  decomposition $P=P_1\cup P_2$ as follows. For $(e,e')\in P$ we have
  $e=\{x,s_{ij}x\}$ and $e'=\{x,s_{pq}x\}$ for some $x\in\Omega(n,k)$
  and $i,p\in I(x)$, $j,q\in J(x)$. We say that $(e,e')\in P_1$ if
  $e,e'$ {``overlap''}, i.e., $i=p$ or $j=q$. Otherwise, if $i\neq
  p$ and $j\neq q$ we say that $(e,e')\in P_2$. Now we can write
  \begin{align*}
   \cB(\rho,\psi)~&=~\cB^{\rm on}(\rho,\psi) + \cB^{{\rm off},1}(\rho,\psi) + \cB^{{\rm off},2}(\rho,\psi)\;, \quad \text{ where }\\
    \cB^{{\rm off},i}(\rho,\psi) ~&=~ 
    \sum\limits_{(e,e')\in P_i}b(e,e')\;,\quad i=1,2\;.
  \end{align*}
  Note that every pair $(e,e')\in P_1$ is part of a unique triangle in
  the graph $(\Omega(n,k),E)$. Indeed, $s_{ij}x$ and $s_{iq}x$
  differ in exactly two coordinates, namely $j$ and $q$. Similarly,
  $s_{ij}x$ and $s_{pj}x$ differ exactly in $i$ and $p$. Moreover,
  every edge $e\in E$ is part of $n-2$ different triangles. Indeed,
  any two neighbors $x,s_{ij}x$ have exactly $n-2$ common neighbors,
  namely the points $s_{iq}x$, $q\in J(x)\setminus\{j\}$ and
  $s_{pj}x$, $p\in I(x)\setminus\{i\}$. Thus we obtain
  \begin{align*}
    \cB^{{\rm off},1}(\rho,\psi)~&=~ 
    \sum\limits_{\triangle}\cB^{\rm off}_\triangle(\rho,\psi)
    ~\geq~ \frac1{2d}\sum\limits_{\triangle}\cA_\triangle(\rho,\psi)
    ~=~ \frac{n-2}{2d}\cA(\rho,\psi)\;,
  \end{align*}
  where we haved summed over all triangle subgraphs $\triangle$ and
  used Lemma \ref{lem:triangle}. Now note that every pair $(e,e')\in
  P_2$ is part of precisely two squares in the graph
  $(\Omega(n,k),E)$. Indeed, if $i\neq p$ and $j\neq q$, the points
  $x,s_{ij}x,s_{pq}s_{ij}x,s_{pq}x$ and the points $x,s_{ij}x,
  s_{iq}x,s_{pq}x$ form a cycle. Thus, using Lemma \ref{lem:square},
  we obtain 
  \begin{align*}
    \cB^{{\rm off},2}(\rho,\psi)~&=~ 
     \frac12\sum\limits_{\Box}\cB^{\rm off}_\Box(\rho,\psi)
    ~\geq~ 0\;,
  \end{align*}
  where we have summed over all square subgraphs $\Box$.

Finally, putting everything together and using Lemma
  \ref{lem:diag} we get
  \begin{align*}
    \cB(\rho,\psi)~&\geq~ \frac2d \cA(\rho,\psi) + \frac{n-2}{2d}\cA(\rho,\psi)
                  ~=~\frac{n+2}{2d}\cA(\rho,\psi)\;. 
  \end{align*}    
\end{proof}

As noted in the introduction, we recover the best known constant in
the modified logarithmic Sobolev inequality as a corollary.

\section{The random transposition model}
\label{sec:randtrapo}

Let $\cS_n$ be the set of permutations of $[n]:=\{1,\dots,n\}$, i.e.,
$\cS_n$ consists of all bijective maps $\sigma : [n] \to [n]$. The
composition $\sigma_1 \circ \sigma_2$ of two permutations $\sigma_1,
\sigma_2 \in \cS_n$ will be denoted by $\sigma_1 \sigma_2$.  For $1
\leq i < j \leq n$ let $\tau_{ij}\in \cS_n$ denote the transposition
which interchanges $i$ and $j$, i.e.,
\begin{align*}
  \tau_{ij}(i) = j\;,\quad \tau_{ij}(j) = i\;,\quad \tau_{ij}(k) = k\quad\forall k\neq i,j\;.
\end{align*}
We define a graph structure on the group $\cS_n$ by saying that two
permutations are neighbors if they differ by precisely one
transposition. Thus every vertex $\sigma\in \cS_n$ has $\binom{n}{2}$
neighbors given by $\{\tau_{ij}\sigma\}_{1\leq i<j\leq n}$, and the
set of edges is $E=\big\{\{\sigma,\tau_{ij}\sigma\}~:~1\leq i<j\leq
n\big\}$.  The simple random walk on $(\cS_n,E_n)$ is given by the
Markov transition rates
\begin{align*}
Q_{\rm RT}(\sigma,\eta)=
 \left\{ \begin{array}{ll}
\frac{2}{n(n-1)}\;,
 & \text{if } \sigma\sim\eta\;,\\
0\;,
 & \text{otherwise },\end{array} \right.
\end{align*}
and the uniform probability measure $\pi_n$ given by
$\pi_n(\sigma)=\abs{\cS_n}^{-1}=(n!)^{-1}$ is reversible for $Q$.

We have the following curvature bound for the random transposition
model.

\begin{theorem}\label{thm:ran-tea}
  The simple random walk $Q_{\rm RT}$ on $\cS_n$ satisfies
  \begin{align*}
     \Ric(Q_{\rm RT}) \geq \frac{4}{n(n-1)}\;.  
  \end{align*}
\end{theorem}

\begin{proof}
  Let us set $d=\frac{n(n-1)}{2}$. 
  Then we need to show that for any
  $\rho\in\cP_*(\cS_n)$ and $\psi\in\R^\cX$ we have
  \begin{align*}
    \cB(\rho,\psi)~\geq~\frac{2}{d}\cA(\rho,\psi)\;.
  \end{align*}
  Let $P=\big\{(e,e')\in E\times E~:~e\sim e', e\neq e'\big\}$ be the
  set of pairs of adjacent non-identical edges. We define a
  decomposition $P=P_1\cup P_2$ as follows. For $(e,e')\in P$ we have
  $e=\{\sigma,\tau_{ij}\sigma\}$ and $e'=\{\sigma,\tau_{pq}x\}$ for
  some $\sigma\in \cS_n$ and $i<j,p<q$. We say that $(e,e')\in P_1$ if
  $e,e'$ do not ``overlap'',
  i.e., $\{i,j\}\cap\{p,q\}=\emptyset$. Otherwise, if
  $\{i,j\}\cap\{p,q\}\neq\emptyset$ we say $(e,e')\in P_2$.
  Now we can write
  \begin{align*}
   \cB(\rho,\psi)~&=~\cB^{\rm on}(\rho,\psi) + \cB^{{\rm off},1}(\rho,\psi) + \cB^{{\rm off},2}(\rho,\psi)\;, \qquad \text{ where }\\
    \cB^{{\rm off},i}(\rho,\psi) ~&=~ \sum\limits_{(e,e')\in P_i}b(e,e')\;,\quad i=1,2\;.
  \end{align*}
  Note that every pair $(e,e')\in P_1$ is part of a unique square in
  the graph $(\cS_n,E)$. Indeed, $\tau_{ij}\sigma$ and
  $\tau_{pq}\sigma$ have the vertex
  $\tau_{pq}\tau_{ij}\sigma=\tau_{ij}\tau_{pq}\sigma$ as their unique
  common neighbor besides $\sigma$. Observe that all pairs of adjacent
  edges in this square belong to $P_1$. Every pair $(e,e')\in P_2$ is
  part of exactly two squares. Indeed, let $(e,e')\in P_2$, and assume
  without loss of generality that $e=\{\sigma,\tau_{ij}\sigma\}$ and
  $e'=\{\sigma,\tau_{iq}\sigma\}$. Then
  $\tau_{iq}\tau_{ij}\sigma=\tau_{jq}\tau_{iq}\sigma$ and
  $\tau_{jq}\tau_{ij}\sigma=\tau_{ij}\tau_{iq}\sigma$ are the two
  distinct common neighbors of $\tau_{ij}\sigma$ and
  $\tau_{iq}\sigma$. Note that all pairs of adjacent edges in these
  squares belong to $P_2$.
  
  For $i=1,2$, we let $A_i$ denote the set of all square subgraphs of
  $(\cS_n,E)$ in which each adjacent pair of edges belongs to
  $P_i$. Thus we obtain
  \begin{align*}
    \cB^{{\rm off},1}(\rho,\psi)~&=~ 
    \sum\limits_{\Box\in A_1}\cB^{\rm off}_\Box(\rho,\psi)
    ~\geq~0\;,\\
    \cB^{{\rm off},2}(\rho,\psi)~&=~ 
    \frac12\sum\limits_{\Box\in A_2}\cB^{\rm off}_\Box(\rho,\psi)
    ~\geq~0\;.
  \end{align*}
  Finally, putting everything together and using Lemma
  \ref{lem:diag} we get
  \begin{align*}
    \cB(\rho,\psi)~&\geq~ \cB^{\rm on}(\rho,\psi)~\geq~\frac2d \cA(\rho,\psi)\;,
  \end{align*}   
  which completes the proof. 
\end{proof}

One might hope that the above lower Ricci bound can be improved by
looking at subgraphs isomorphic to $\cS_3$ taking over the role of the
triangles in the Bernoulli--Laplace model. However, the next lemma
shows that they only give a nonnegative contribution to the
off-diagonal $\cB$-term in general, which does not improve the bound
obtained in Theorem \ref{thm:ran-tea}.

\begin{lemma}\label{lem:BoffsharpS3}
  For all $\rho: \cS_3\to\R_+$ and $\psi:\cS_3\to\R$ we have
  \begin{align*}
    \cB^{\rm off}(\rho,\psi)~\geq~0\;.
  \end{align*}
  Moreover, this bound is sharp in the sense that for any $\kappa > 0$
  there exist $\rho$ and $\psi$ such that
  $\cB^{\rm off}(\rho,\psi) < \kappa \cA(\rho, \psi)$.
\end{lemma}

\begin{proof}
  The non-negativity follows from writing $\cB^{\rm off}$ as a sum of
  contributions from squares and using Lemma \ref{lem:square}. To see
  that this bound is sharp, define $\rho_\eps$ and $\psi$ by
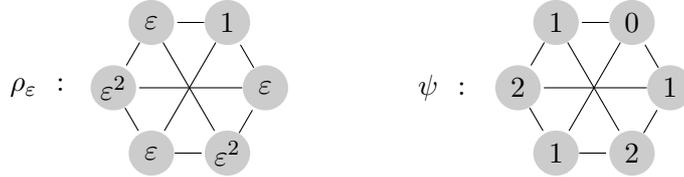
\begin{center}
\begin{tikzpicture}[shorten >=1pt,->]
  \tikzstyle{label}=[circle,fill=white!20,minimum size=17pt,inner sep=0pt]
\node[label,xshift=4cm,yshift=.5cm]  {$\rho_\eps$ \ :};

  \tikzstyle{vertex}=[circle,fill=black!20,minimum size=17pt,inner sep=0pt]
  \foreach \name/\angle/\text in {P-1/360/\eps, P-2/60/1, P-3/120/\eps, 
                                  P-4/180/\eps^2, P-5/240/\eps, P-6/300/\eps^2}
    \node[vertex,xshift=6cm,yshift=.5cm] (\name) at (\angle:1cm) {$\text$};

  \foreach \from/\to in {1/2,2/3,3/4,4/5,5/6,6/1,1/4,2/5,3/6}
	{ \draw[-] (P-\from) edge (P-\to); }
\end{tikzpicture}
\qquad\qquad
\begin{tikzpicture}[shorten >=1pt,->]
  \tikzstyle{label}=[circle,fill=white!20,minimum size=17pt,inner sep=0pt]
\node[label,xshift=4cm,yshift=.5cm]  {$\psi$ \ :};

  \tikzstyle{vertex}=[circle,fill=black!20,minimum size=17pt,inner sep=0pt]
  \foreach \name/\angle/\text in {P-1/360/1, P-2/60/0, P-3/120/1, 
                                  P-4/180/2, P-5/240/1, P-6/300/2}
    \node[vertex,xshift=6cm,yshift=.5cm] (\name) at (\angle:1cm) {$\text$};

  \foreach \from/\to in {1/2,2/3,3/4,4/5,5/6,6/1,1/4,2/5,3/6}
	{ \draw[-] (P-\from) edge (P-\to); }
	
\end{tikzpicture}
\end{center} Then one can check that as $\eps\to 0$,
  \begin{align*}
    \frac{\cB^{\rm off}(\rho_\eps,\psi)}{\cA(\rho_\eps,\psi)}~\to~0\;.
  \end{align*}
\end{proof}

\bibliographystyle{plain}
\bibliography{ricci}

\begin{thebibliography}{10}

\bibitem{AGS11b}
L.~Ambrosio, N.~Gigli, and G.~Savar{\'e}.
\newblock Metric measure spaces with {R}iemannian {R}icci curvature bounded
  from below.
\newblock {\em Duke Math. J.}, 163(7):1405--1490, 2014.

\bibitem{BE85}
D.~Bakry and M.~{\'E}mery.
\newblock Diffusions hypercontractives.
\newblock In {\em S\'eminaire de probabilit\'es, {XIX}, 1983/84}, volume 1123
  of {\em Lecture Notes in Math.}, pages 177--206. Springer, Berlin, 1985.

\bibitem{BT06}
S.~G. Bobkov and P.~Tetali.
\newblock Modified logarithmic {S}obolev inequalities in discrete settings.
\newblock {\em J. Theoret. Probab.}, 19(2):289--336, 2006.

\bibitem{CDPP09}
P.~Caputo, P.~Dai~Pra, and G.~Posta.
\newblock Convex entropy decay via the {B}ochner-{B}akry-{E}mery approach.
\newblock {\em Ann. Inst. Henri Poincar\'e Probab. Stat.}, 45(3):734--753,
  2009.

\bibitem{CHLZ11}
S.-N. Chow, W.~Huang, Y.~Li, and H.~Zhou.
\newblock Fokker-{P}lanck equations for a free energy functional or {M}arkov
  process on a graph.
\newblock {\em Arch. Ration. Mech. Anal.}, 203(3):969--1008, 2012.

\bibitem{DSC96}
P.~Diaconis and L.~Saloff-Coste.
\newblock Logarithmic {S}obolev inequalities for finite {M}arkov chains.
\newblock {\em Ann. Appl. Probab.}, 6(3):695--750, 1996.

\bibitem{DiSh81}
P.~Diaconis and M.~Shahshahani.
\newblock Generating a random permutation with random transpositions.
\newblock {\em Z. Wahrsch. Verw. Gebiete}, 57(2):159--179, 1981.

\bibitem{DiSh87}
P.~Diaconis and M.~Shahshahani.
\newblock Time to reach stationarity in the {B}ernoulli-{L}aplace diffusion
  model.
\newblock {\em SIAM J. Math. Anal.}, 18(1):208--218, 1987.

\bibitem{EM12}
M.~Erbar and J.~Maas.
\newblock Ricci curvature of finite {M}arkov chains via convexity of the
  entropy.
\newblock {\em Arch. Ration. Mech. Anal.}, 206(3):997--1038, 2012.

\bibitem{EM13}
M.~Erbar and J.~Maas.
\newblock Gradient flow structures for discrete porous medium equations.
\newblock {\em Discrete Contin. Dyn. Syst.}, 34(4):1355--1374, 2014.

\bibitem{GQ03}
F.~Gao and J.~Quastel.
\newblock Exponential decay of entropy in the random transposition and
  {B}ernoulli-{L}aplace models.
\newblock {\em Ann. Appl. Probab.}, 13(4):1591--1600, 2003.

\bibitem{GM12}
N.~Gigli and J.~Maas.
\newblock Gromov-{H}ausdorff convergence of discrete transportation metrics.
\newblock {\em SIAM J. Math. Anal.}, 45(2):879--899, 2013.

\bibitem{Goe04}
S.~Goel.
\newblock Modified logarithmic {S}obolev inequalities for some models of random
  walk.
\newblock {\em Stochastic Process. Appl.}, 114(1):51--79, 2004.

\bibitem{GMPRST13}
N.~Gozlan, J.~Melbourne, W.~Perkins, C.~Roberto, P-M. Samson, and P.~Tetali.
\newblock {Working Group in New directions in mass transport: discrete versus
  continuous}.
\newblock {\em AIM SQuaRE report, October}, 2013.

\bibitem{GRST13}
N.~Gozlan, C.~Roberto, P.-M. Samson, and P.~Tetali.
\newblock Displacement convexity of entropy and related inequalities on graphs.
\newblock {\em Probability Theory and Related Fields}, 160:47--94, 2014.

\bibitem{LY98}
T.-Y. Lee and H.-T. Yau.
\newblock Logarithmic {S}obolev inequality for some models of random walks.
\newblock {\em Ann. Probab.}, 26(4):1855--1873, 1998.

\bibitem{LV09}
J.~Lott and C.~Villani.
\newblock Ricci curvature for metric-measure spaces via optimal transport.
\newblock {\em Ann. Math. (2)}, 169(3):903--991, 2009.

\bibitem{Ma11}
J.~Maas.
\newblock Gradient flows of the entropy for finite {M}arkov chains.
\newblock {\em J. Funct. Anal.}, 261(8):2250--2292, 2011.

\bibitem{Mie11a}
A.~Mielke.
\newblock A gradient structure for reaction-diffusion systems and for
  energy-drift-diffusion systems.
\newblock {\em Nonlinearity}, 24(4):1329--1346, 2011.

\bibitem{Mie11b}
A.~Mielke.
\newblock Geodesic convexity of the relative entropy in reversible {M}arkov
  chains.
\newblock {\em Calc. Var. Partial Differential Equations, Online first}, 2012.

\bibitem{Oll09}
Y.~Ollivier.
\newblock Ricci curvature of {M}arkov chains on metric spaces.
\newblock {\em J. Funct. Anal.}, 256(3):810--864, 2009.

\bibitem{Oll10}
Y.~Ollivier.
\newblock A survey of {R}icci curvature for metric spaces and {M}arkov chains.
\newblock {\em Adv. Stud. Pure Math.}, 57:343--381, 2010.

\bibitem{Sam05}
M.D. Sammer.
\newblock {Aspects of mass transportation in discrete concentration
  inequalities}.
\newblock {\em PhD thesis, {G}eorgia {I}nstitute of {T}echnology}, 2005.

\bibitem{S06}
K.-Th. Sturm.
\newblock On the geometry of metric measure spaces. {I} and {II}.
\newblock {\em Acta Math.}, 196(1):65--177, 2006.

\end{thebibliography}

\end{document}